\documentclass[runningheads]{llncs}
\usepackage{graphicx}
\usepackage{float}
\usepackage{mathtools}
\usepackage{wrapfig}
\usepackage{amssymb, amsmath, latexsym}
\usepackage{nicefrac}
\usepackage{hhline}
\usepackage{multirow}
\usepackage{pifont}
\usepackage[colorinlistoftodos,bordercolor=orange,backgroundcolor=orange!20,linecolor=orange,textsize=scriptsize]{todonotes}
\usepackage{algorithm}\usepackage{algpseudocode}
\newcommand{\argmin}{\mathop{\arg\!\min}}

\def \R {\mathbb R}

\newcommand{\EndProof}{\begin{flushright}$\square$\end{flushright}}

\newcommand{\EE}{\mathbf{E}}

\def\R{\mathbb{R}}

\def\R{\mathbb R}

\def\EE{\mathbb E}

\begin{document}
\title{Linearly Convergent Gradient-Free Methods for Minimization of Parabolic Approximation\thanks{
The research of A. Beznosikov and A. Gasnikov was supported by Russian Science Foundation (project No. 21-71-30005). This work was partially conducted while A. Beznosikov was on the project internship in Sirius University of Science and Technology.}}
\titlerunning{Linearly Convergent Methods for Quadratic Approximation}
%
\author{
Aleksandra Bazarova\inst{1}
Aleksandr Beznosikov\inst{1,2}\and
Alexander Gasnikov\inst{1,3,4}}
\authorrunning{A. Bazarova, A. Beznosikov  and A. Gasnikov}
%
\institute{Moscow Institute of Physics and Technology, Russia \and
Higher School of Economics, Russia \and
Institute for Information Transmission Problems RAS, Russia  \and
Caucasus Mathematical Center, Adyghe State University, Russia}
\maketitle              
\begin{abstract}
Finding the global minimum of non-convex functions is one of the main and most difficult problems in modern optimization. 
In the first part of the paper, we consider a certain class of "good"{} non-convex functions that can be bounded above and below by a parabolic function. We show that using only the zeroth-order oracle, one can obtain the linear speed $\log \left(\nicefrac{1}{\varepsilon}\right)$ of finding the global minimum on a cube. The second part of the paper looks at the nonconvex problem in a slightly different way. We assume that minimizing the quadratic function, but at the same time we have access to a zeroth-order oracle with noise and this noise is proportional to the distance to the solution. Dealing with such noise assumptions for gradient-free methods is new in the literature. We show that here it is also possible to achieve the linear rate of convergence.

\end{abstract}

\section{Introduction}


Methods for minimizing convex functions are well researched in the literature \cite{boyd_vandenberghe_2018,NoceWrig06} and have good guarantees of convergence to a solution. When the objective function is non-convex, the problem becomes much more complicated. Meanwhile, the ability to find the global minimum of non-convex functions is an equally important issue, but in general, this is NP-hard. The main idea of constructing an analysis around non-convex functions is the introduction of some restrictions on the problem: these can be desire to search not for a global minimum, but only for a local minimum (in the hope that local is good enough) or restrictions on a function on a set of optimization. 

\textbf{In the first part} of the paper, we follow the same way and try to find a global minimum of the function bounded by two parabolic functions 
More formally, our statement of the problem can be described as follows:
\begin{equation}
    \label{min}
    \min_{x \in C} f(x),
\end{equation}
where set $C$ is a cube in $\R^d$, i.e. for all $x \in C$: $l_i \leq x_i \leq u_i$ with $i$ from $1$ to $d$. We do not know whether $f(x)$ is convex, smooth, whether its gradient is bounded or not. In general, the function can be any, including non-convex and non-differentiable. But we assume that the function $f (x)$ satisfies the following condition for all $x \in C$:
\begin{equation}
    \label{good}
    \frac{\mu}{2} \|x - x^* \|^2 \leq f(x) - f(x^*) \leq \frac{L}{2} \|x - x^* \|^2.
\end{equation}
Hereinafter, $x^*$ is the solution to problem \eqref{min} and we use the ordinary Euclidean norm $\| \cdot \|$. Inequalities \eqref{good} define the "good" class. Such condition describes a rather large set of functions that has a global minimum on the cube: \\
\begin{figure}[h]
\centering
\includegraphics[width =  0.96\linewidth]{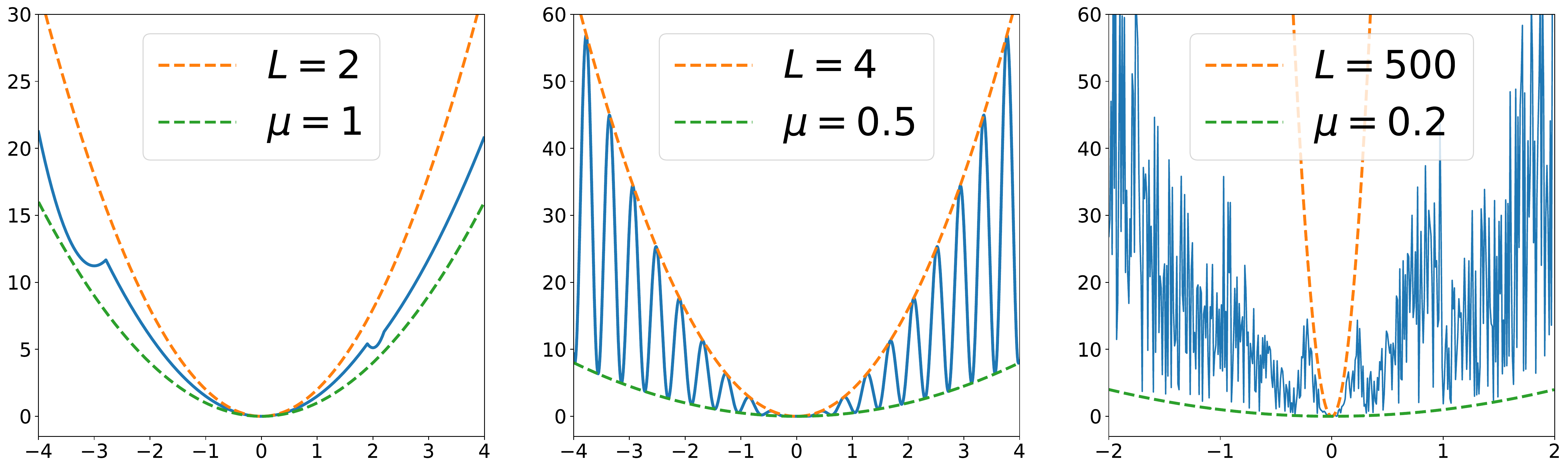}
\caption{Examples of functions that satisfy condition \eqref{good} with different constants $L$ and $\mu$. From left to right, the ratio $\nicefrac{L}{\mu}$ increases.}
\label{fig:1}
\end{figure}

One can note that using a first-order oracle (gradient) for such functions is not a good idea. Since, due to possible large and sharp oscillations, the gradient does not carry any useful global information. Thus, the methods outlined in this paper rely exclusively on the zeroth-order oracles. 

It seems natural that if the constant $L$ is too large or/and the constant $\mu$ is too small, then the search for the solution becomes more difficult. Therefore, we propose another class of "very good"{} functions for which the constants $L$ and $\mu$ differ but not much:
\begin{equation}
    \label{good_good}
    f(x) - f(x^*) = \left(\frac{M}{2}  + \delta(x) \right) \| x - x^*\|^2_2, ~~\text{with } |\delta(x)| \leq \Delta = \frac{M}{16(d-1)},
\end{equation}
for all $x \in C$.
Such functions are quite quadratic, but they can fluctuate with a level of deviation equal to $\delta(x)$. It is easy to see that the condition \eqref{good} is satisfied with $L = \nicefrac{M}{2} + \nicefrac{M}{16(d-1)}$ and $\mu = \nicefrac{M}{2} - \nicefrac{M}{16(d-1)}$.

\textbf{In the second part} of the paper we look at the "good" functions described above in other way: let us be given a parabolic function, but the oracle returns not the exact value of this function, but with noise:
$$f(x) =  \tilde L \|x-x^*\|^2 + \xi(x) + \delta(x).$$
Here, $\xi$ is responsible for stochastic (random) noise, and $\delta$ -- for deterministic noise.
Moreover, for the functions \eqref{good} and \eqref{good_good} $\xi = 0$, and $\delta(x) \sim \|x-x^* \|^2$. In this part of the work we consider a slightly different concept, namely
\begin{equation}
    \label{good_bad}
    f(x, \xi) = \frac{1}{2}(x-x^*)^TA (x-x^*) + (\xi + \delta(x)) \|x-x^* \|,
\end{equation}
where $A \succ 0$. Having only this information about the function, we want to solve the problem:
\begin{equation}
    \label{min_q}
    \min_{x \in \R^d} \frac{1}{2}(x-x^*)^TA (x-x^*). 
\end{equation}
For this one can reconstruct the real gradient using finite differences:
\begin{equation}
\label{grad_fd}
\frac{d}{2\tau} (f(x+\tau e, \xi^+) - f(x-\tau e, \xi^-))e,
\end{equation}
where $e$ -- some random vector uniformly distributed on the Euclidean sphere. Then such a gradient approximation can be used in Gradient Descent, which it does.

\subsection{Our contribution and related works}


Let's start with a discussion with ideas of \textbf{the first part}.

There are already results in the literature where minimization of some specific class of non-convex functions is considered. In \cite{shor2012minimization,polyak1987introduction}, the objective function is non-convex, but at the same time, it is bounded from below and above by some "good"{} functions. Essentially, there is a similar approach in \cite{NIPS2015_5841}, they consider gradient-free minimization of convex functions, but additionally, assume that the zeroth-order oracle takes values of the function with noise. This concept is suitable for non-convex problems in which the objective function "oscillates" around some convex function.

The idea of our method is remotely similar to the simplest zeroth-order methods for minimizing one-dimensional unimodal functions: we calculate the value of the function at some points, and then, using this information, we decrease the optimization set by a certain number of times.

Also, our methods are partly close to the Monte Carlo type algorithms \cite{book}. These methods are also suitable for non-convex optimization problems and exploit the idea of Markov search for a solution. However, they also require the problem to be "good" enough. Our methods also do some kind of search, which is not stochastic but simply uses information about the "goodness" of the function. It is interesting to note that one of our methods has exponential growth depending on the dimension $d$, as well as ones from \cite{book}. 


We propose an algorithm for finding the global minimum on a cube for the function \eqref{good}. It requires $\log \left(\nicefrac{1}{\varepsilon}\right)$ iterations and at each iteration the zeroth-order oracle is called $\mathcal{O}\left(\left(\nicefrac{Ld}{\mu}\right)^d\right)$ times. The main idea of this algorithm is that we split a large cube into many small cubes and calculate the function value in each of them. Then we find the minimum value among all the cubes. It can be shown that the real minimum of the problem lies not far from the found point. Therefore, the edge of the original cube can be cut in half, and we can consider a new cube with the center at the found point. 
Such an algorithm is specifically capable in practice in low-dimensional problems, where the ratio $\nicefrac{L}{\mu}$ can be quite large. See details in Section \ref{bad}.

For class of functions \eqref{good_good}, we propose a less demanding algorithm, it also has a linear convergence rate $\log \left(\nicefrac{1}{\varepsilon}\right)$, but the complexity of its iteration is only $\mathcal{O}(d)$. In this case, at each subiteration of the algorithm, we take one of the variables $x_i$, and equate the rest ($x_j$ with $j \neq i$) with the average value and fix it. And for the variable $x_i$, we request the value of the function at $n$ points, uniformly distributed from $l_i$ to $u_i$. Next, we find the minimum among these $n$ points. Then the $i$th edge of the cube can be halved, and we can consider a new edge centered at the found minimum point. This algorithm shows itself well in practice and for large-scale problems. For more details see Section \ref{goood}.

\vspace{\baselineskip}

The approach suggested in the first part is not very popular in modern research. \textbf{The second part}  follows current trends and works with the concept of gradient reconstruction through finite differences, which well studied in literature \cite{Shamir15,Nesterov}. 

As mentioned above, the idea of an inexact oracle is also used here.  It is important to note that there are two different random variables $\xi^+$ and $\xi^-$ in \eqref{grad_fd}. Approximations of this type are referred to as one-point feedback \cite{akhavan2020exploiting,gasnikov2017stochastic,zhang2020improving,novitskii2021improved} (compare to two-point feedback from \cite{aleks2020gradientfree,Beznosikov}, where $\xi^+ = \xi^-$). The concept of a one-point feedback is less friendly from the point of view of theoretical analysis, but more realizable from the point of view of practice, since in a real problem it is difficult to achieve the value of a function in two different points with the same realization of a random variable.

In the analysis that is present in the literature (for one or two-point feedbacks) \cite{NIPS2016_186fb23a,akhavan2020exploiting,Beznosikov,gasnikov2017stochastic}, it is assumed that the noise (or its second moment) is uniformly bounded. In our setting, the noise depends on the distance to the solution -- this is the main novelty of our problem statement. 

We show that under certain conditions on the noise level and the correct choice of the parameters $\tau$ in \eqref{grad_fd} and the step of the Gradient Descent, it is also possible to achieve a linear convergence rate.

\section{"Good"{} functions} \label{bad}

In this section we concentrate on functions from \eqref{good}. For a better understanding of the method, we present a sketch of the analysis of the algorithm for the one-dimensional case.

\subsection{Method intuition on a segment}

Suppose we have a function $f: [l;u] \to \R$ and it satisfies condition \eqref{good}. Then let consider the following procedure, which we called {\tt Bad Binary Search} (or {\tt BBS}\\
\begin{minipage}{0.45\textwidth}
     \begin{algorithm}[H]
\caption{{\tt BBS}}
	\label{alg1}
\begin{algorithmic}
\State 
\noindent {\bf Input:} Accuracy $\varepsilon$, parameters $L$, $\mu$ from \eqref{good} and bounds $l$, $u$.
\State Let $b := l$, $B := u$ and $n :=2\left\lceil\sqrt{\frac{L}{\mu}}\right\rceil$.
\While {$B - b \geq 2\varepsilon$}
    \begin{eqnarray*}
    &&i^* := \argmin_{i \in \{0, \ldots, n\}} f\left(b + i \cdot \nicefrac{(B-b)}{n}\right),\\
    &&b := \max\left(b; ~b + \left(i^* - \nicefrac{n}{4}\right) \cdot \nicefrac{(B-b)}{n}\right), \\
    &&B := \min\left(B;~b + \left(i^* + \nicefrac{n}{4}\right) \cdot \nicefrac{(B-b)}{n}\right).
    \end{eqnarray*}
\EndWhile
\State 
\noindent {\bf Output:} $\nicefrac{(B-b)}{2}$.
\end{algorithmic}
\end{algorithm}
\end{minipage}
\begin{minipage}{0.05\textwidth}
\end{minipage}
\begin{minipage}{0.52\textwidth}
       for short). The essence of this procedure is very simple. At each iteration of the algorithm, we divide the current segment into $n$ parts and calculate the value of the function at the ends of these segments ($n + 1$ calculations in total). Next, we find the minimum among these $n + 1$ values. It seems that the found point should lie somewhere close to the solution, but this is not entirely true -- it depends on $n$. We claim that it is possible to choose such $n$ that the real minimum will lie in the vicinity of the found minimum. For this we turn to Figure 2: the blue line corresponds to the values of the function, the orange and green lines are the bounding
\end{minipage}
parabolas, the algorithm calculates the value of the function at black points, the minimum level is reached at the point $x^{min}$, $x^*$ is the real minimum of the function, and the point $x^{cl}$ is the closest point to it. 
\begin{figure}[h]
\centering
\includegraphics[width =  0.86\linewidth]{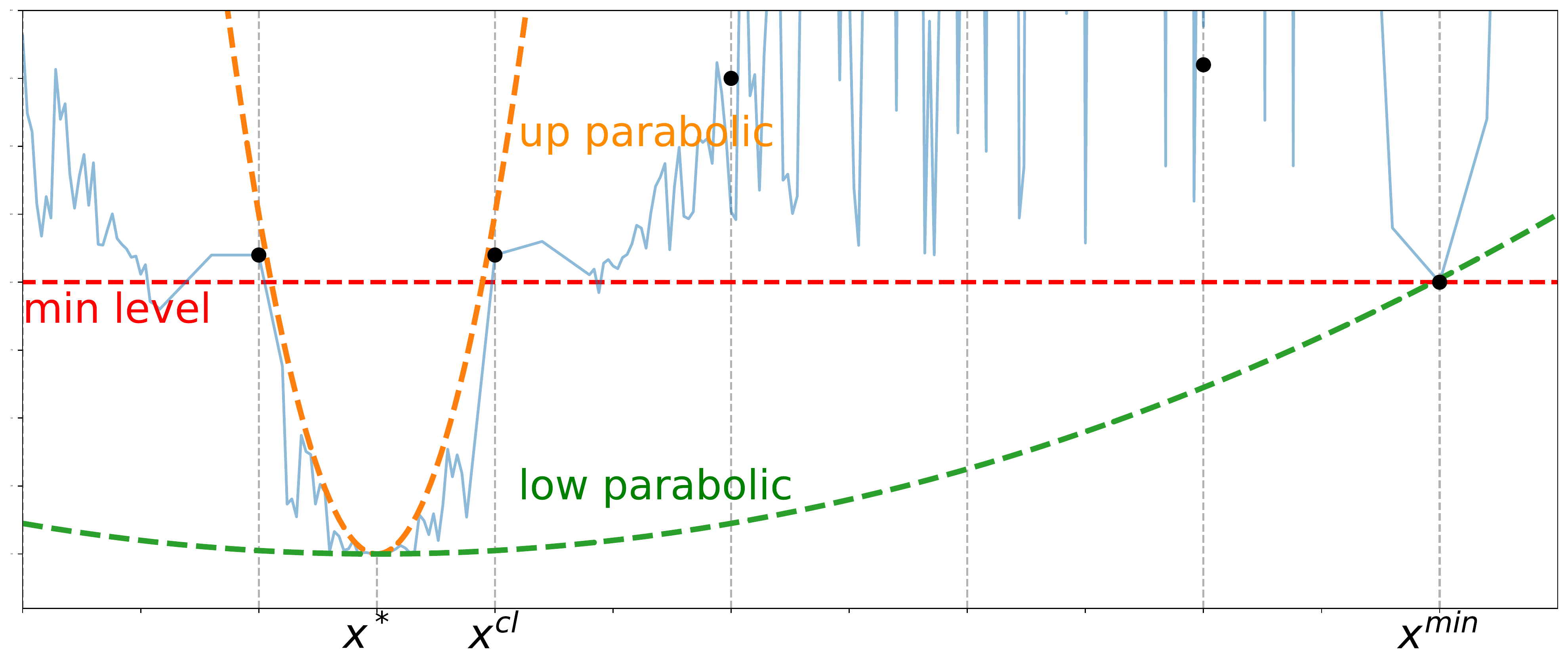}
\caption{An illustration of the reasoning about the correctness of Algorithm \ref{alg1}. Read the description of notations in the text}
\label{fig:2}
\end{figure}\\
We want $\nicefrac{|x^{min} - x^*|}{u - l} \leq \nicefrac{1}{4}$ (where $l,u$ -- current bounds of the segment), which means that we can essentially cut our segment in half and consider a new segment centered at the found minimum. 

Since the value at the point $x^{min}$ is the minimum among all the others, then
\begin{equation*}
    \frac{\mu}{2} (x^{min} - x^*)^2 \leq f(x^{min}) \leq f(x^{cl}) \leq \frac{L}{2} (x^{cl} - x^*)^2.
\end{equation*}
Next, we use the facts that the length of the small segment is $\nicefrac{u-l}{n}$, and the distance between $x^*$ and $x_{cl}$ is no more than $\nicefrac{u-l}{2n}$:
\begin{equation*}
    \frac{\mu}{2} (x^{min} - x^*)^2 \leq \frac{L}{2} (x^{cl} - x^*)^2  \leq \frac{L}{8n^2} (u-l)^2.
\end{equation*}
From where we instantly get 
\begin{equation*}
\frac{|x^{min} - x^*|}{u - l} \leq \frac{1}{2n} \sqrt{\frac{L}{\mu}} \leq \frac{1}{4}.
\end{equation*}
The last inequality follows from the requirement $\nicefrac{|x^{min} - x^*|}{u - l} \leq \nicefrac{1}{4}$. Then we get a lower bound on $n$:
\begin{equation*}
n \geq 2 \left\lceil\sqrt{\frac{L}{\mu}}\right\rceil.
\end{equation*}

\textbf{Remark.} Note that in Algorithm \ref{alg1} it is necessary to divide $n$ by $4$, but $n$ may not be divisible by $4$. This does not violate the convergence of the method, but if we take $n$ is a multiple of 4, then it turns out that at the next iteration of the algorithm we already know the value of the function at half the points, since they coincided with the points from the previous iteration.

\subsection{Theoretical convergence in $\mathbb{R}^d$}

In this part of the work, we present and analyze the algorithm in the case when we work in a space of dimension $d$. The multidimensional analog of Algorithm \ref{alg1} in the following way:\\
\begin{minipage}{1\textwidth}
     \begin{algorithm}[H]
\caption{{\tt Multi BBS}}
	\label{alg2}
\begin{algorithmic}
\State 
\noindent {\bf Input:} Accuracy $\varepsilon$, parameters $L$, $\mu$ from \eqref{good} and bounds $\vec{l} = (l_1, \dots l_d)$, $\vec{u} = (u_1, \dots u_d)$.
\State Let $n := \alpha \left\lceil\sqrt{\nicefrac{dL}{\mu}}\right\rceil$,\\
$b :=  l$ and $B :=  u$.
\While {$\|B - b \| \geq \varepsilon$}
    \begin{eqnarray*}
    r &:=& \max_{i \in \{1,\ldots d \}} \nicefrac{(B_i - b_i)}{n}\\
    S_j &:=& \{0, 1,\ldots, \lfloor \nicefrac{(B_j - b_j)}{r} \rfloor\}, ~~~~ S := S_1 \times \ldots \times S_d,\\
    i^* &:=& \argmin_{i \in S} f\left(b + i \cdot r \right),\\
    b &:=& \left\{b_j := \max\left(b_j; ~b_j + \left(i^*_j - \nicefrac{n}{2\alpha}\right) \cdot r\right), ~ j \in \{1,\ldots d\} \right\}, \\
    B &:=&  \left\{B_j := \min\left(B_j; ~B_j + \left(i^*_j + \nicefrac{n}{2\alpha}\right) \cdot r\right), ~ j \in \{1,\ldots d\} \right\}.
    \end{eqnarray*}
\EndWhile
\State 
\noindent {\bf Output:} $\nicefrac{B-b}{2}$.
\end{algorithmic}
\end{algorithm}
\end{minipage}

This algorithm is more complicated than Algorithm \ref{alg1}, this is due to the fact that we are working with a cube, moreover, this cube may have edges of unequal length. Also, here we generalize the approach for picking $n$ and introduce the parameter $\alpha > 1$. 

The idea of {\tt Multi BBS} repeats the idea of the one-dimensional {\tt BBS} algorithm. We also split our cube into small pieces, calculate values at the points on these pieces, and move on to a new smaller cube centered at the minimum of the selected values. It is important to note the size of the pieces into which the original cube is split: all the edges of each small cube have the same length, and this length is determined by the length of the longest edge of the original cube (see the first line of the main loop). This approach allows us to optimize variables that have a large spread (long edge) first.
Thereby if one of the edges of the original cube were much larger than the other ones, then, in fact, only the variable responsible for this edge would be the one to be optimized. The described strategy equalizes the sizes of all cube edges fast, and this is better than splitting large cube edges into small pieces at once. 

\begin{theorem}
{\tt Multi BBS} algorithm with $\alpha > 1$ converges to the global minimum of the function \eqref{good}. Moreover, the value equal to the maximum length of the cube edge decreases by at least $\alpha$ times at each iteration. 
\end{theorem}

\begin{proof}
Let's introduce the same notation as in the previous subsection: $x_{min}$ is the point with the minimum value among the selected, $x^*$ is the real minimum of the function, and the point $x_{cl}$ is the closest point to it. Then
\begin{equation*}
    \frac{\mu}{2} \sum\limits_{i = 1}^d (x^{min}_i - x_i^*)^2 \leq f(x^{min}) \leq f(x^{cl}) \leq \frac{L}{2} \sum\limits_{i = 1}^d (x^{cl}_i - x_i^*)^2.
\end{equation*}
Since the edge of a small cube is at most $r$, then for all $i$
\begin{equation*}
(x^{cl}_i - x_i^*)^2 \leq  \frac{r^2}{4},
\end{equation*}
and 
\begin{equation*}
    \frac{\mu}{2} \sum\limits_{i = 1}^d (x^{min}_i - x_i^*)^2 \leq \frac{dLr^2}{8} = \frac{dL}{8n^2} \left(\max_{i \in \{1,\ldots d \}} [B_i - b_i]\right)^2.
\end{equation*}
It is easy to see that with $n = \alpha \left\lceil\sqrt{\nicefrac{dL}{\mu}}\right\rceil$, we get
\begin{equation*}
    \sum\limits_{i = 1}^d (x^{min}_i - x_i^*)^2  \leq \frac{1}{4\alpha^2} \left(\max_{i \in \{1,\ldots d \}} [B_i - b_i]\right)^2.
\end{equation*}
Whence it follows that for all $i$
\begin{equation*}
    \frac{|x^{min}_i - x_i^*|}{\max_{i \in \{1,\ldots d \}} [B_i - b_i]}  \leq \frac{1}{2\alpha} .
\end{equation*}
This inequality ensures that the maximum edge length of the new cube is (at least) $\alpha$ times less than the maximum edge length of the old one.
\EndProof
\end{proof}

The theorem implies the following corollary on the complexity of the algorithm:

\begin{corollary}
Algorithm \ref{alg2} requires $\mathcal{O}\left(\log_{\alpha}\left(\nicefrac{1}{\varepsilon}\right)\right)$ iterations to find a solution (in terms of Algorithm 2). Moreover, the oracle complexity of each iteration is $\mathcal{O}\left(\left(\alpha \sqrt{\nicefrac{dL}{\mu}}\right)^d\right)$.
\end{corollary}
\begin{proof}

To prove the first statement, we write the simple chain for $B$ and $b$ after $T$ iteration:
\begin{equation*}
    \|B^T - b^T \|^2 \leq d \left(\max_{i \in \{1,\ldots d \}} B^T_i - b^T_i \right)^2 \leq d \left(\frac{ \max_{i \in \{1,\ldots d \}} [u_i - l_i]}{\alpha^T} \right)^2 \leq \varepsilon^2,
\end{equation*}
where $u$ and $l$ -- starting cube boundaries. This implies the required statement. 

The second statement follows from the fact that in the worst case (when all the edges of the cube are equal) we need to calculate the value of the function at $\mathcal{O}(n^d)$ points.
\EndProof
\end{proof}
\textbf{Remark.} The complexity of one iteration increases dramatically with the growth of the dimension; therefore, this method is proposed to be used for solving low-dimensional problems.

\subsection{Small dimension numerical experiments}

We give examples of how the algorithm works on low-dimensional problems: one-dimensional and two-dimensional. 
First, consider the following function:
\begin{equation}
\label{exp_f_1}
f(x) = 10 (x - 2)^2 - 4 \cos[17(x-2)] + 4, ~~~~ x \in [0, 6.5].
\end{equation}
The global minimum on this segment is the point $x^* = 2$. We take $L = 600$, $\mu = 10$. Starting point is the center of the segment $x^0 = 3.25$. {\tt Multi BBS} algorithm starts with $\alpha = 1.5, 2, 3, 4$. The convergence of the algorithm is shown in Figure \ref{fig:3} (a). 

Next, we work with a 2-dimensional problem -- the Levy function:
\begin{eqnarray}
\label{exp_f_2}
f(x,y) &=& \sin^2[3\pi (x- 2.7)] + (x-3.7)^2 (1 + \sin^2[3\pi (y - 0.3)]) \\
&&+ (y-1.3)^2(1 + \sin^2[2\pi (y - 0.3)]),
\end{eqnarray}
where $x,y \in [-10, 10]$. The optimal point: $f(3.7, 1.3) = 0$ and $L = 150$, $\mu = 1$. For the trajectory of convergence, see Figure \ref{fig:3} (b).\\
\begin{figure}[h]
\includegraphics[width =  \linewidth]{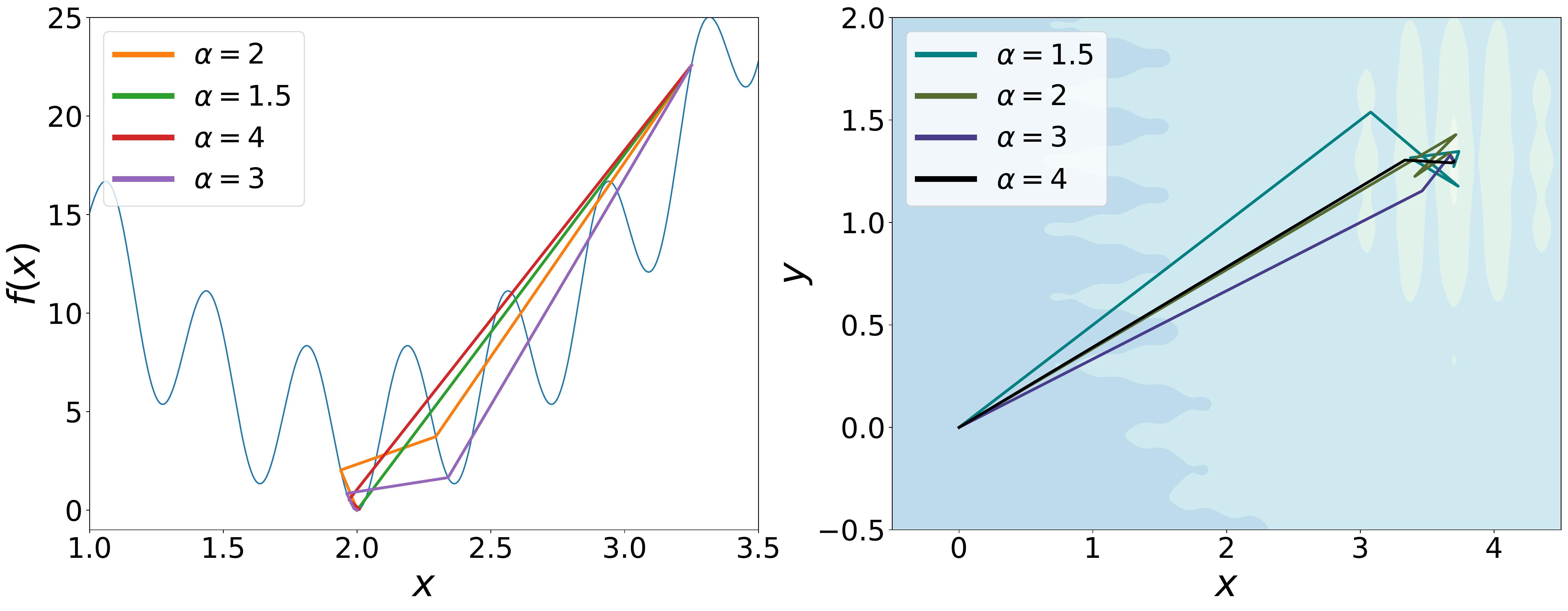}
\begin{minipage}{0.5\textwidth}
\begin{center}
~~~~~(a)
\end{center}
\end{minipage}
\begin{minipage}{0.5\textwidth}
\begin{center}
~~~~~~(b)
\end{center}
\end{minipage}
\caption{Convergence of the {\tt Multi BBS} algorithm for problems of dimensions 1 and 2: (a) problem \eqref{exp_f_1}, (b) problem \eqref{exp_f_2}.}
\label{fig:3}
\end{figure}

\section{Very "good"{} functions} \label{goood}

In this section we analyze functions $f(x)$ which meet the requirements \eqref{good_good}. For such a problem statement, we present a less demanding version of the {\tt BBS} algorithm:\\
\begin{minipage}{1\textwidth}
     \begin{algorithm}[H]
\caption{{\tt Direction BBS}}
	\label{alg3}
\begin{algorithmic}
\State 
\noindent {\bf Input:} Accuracy $\varepsilon$, parameters $M$, $\Delta$ from \eqref{good_good} and bounds $l = (l_1, \dots l_d)$, $u = (u_1, \dots u_d)$.
\State Let $n := 15$,
\State $b :=  l$, $B :=  u$ and $m := \nicefrac{l + u}{2}$.
\While {$\|B - u\| \geq 2\varepsilon$}
    \For {$i = 1,\ldots d$}
    \begin{eqnarray*}
    R &:=& \max_{i \in \{1,\ldots d \}} B_i - b_i\\
    j^* &:=& \argmin_{j = 0, \ldots, n} f\left(m_1,\ldots, m_{i-1}, b_i + j \cdot \nicefrac{(B_i-b_i)}{n}, m_{i+1}, \ldots m_d\right),\\
    m_i &:=& b_i + j^* \cdot \nicefrac{(B_i-b_i)}{n} \\
    b_i &:=& \max\left(b_i;~m_i - \nicefrac{R}{3}\right), \\
    B_i &:=& \min\left(B_i;~m_i + \nicefrac{R}{3}\right).
    \end{eqnarray*}
    \EndFor
\EndWhile
\State 
\noindent {\bf Output:} $\nicefrac{B - b}{2}$.
\end{algorithmic}
\end{algorithm}
\end{minipage}

This algorithm no longer draws a "grid" over the entire cube. In this case, at  each external iteration (while), we go through all the variables in turn. At the inner iteration (for), we consider only one variable, while the rest are fixed equal to $m_j$. Then we do the procedure in a similar way to {\tt BBS} -- we divide the current edge and do a one-dimensional search for the minimum among the calculated points. Due to the fact that the problem is very "good"{}, this approach allows the method to converge. This is what the next section is about.

\subsection{Theoretical analysis}

\begin{theorem}
{\tt Direction BBS} algorithm converges to the global minimum of the function \eqref{good_good}. Moreover, the value equal to the maximum length of the cube edge decreases by at least $\nicefrac{3}{2}$ times at each main iteration (an outer loop with while).
\end{theorem}

\begin{wrapfigure}[15]{r}{0.33\linewidth} 
\vspace{-3ex}
\includegraphics[width =  1\linewidth]{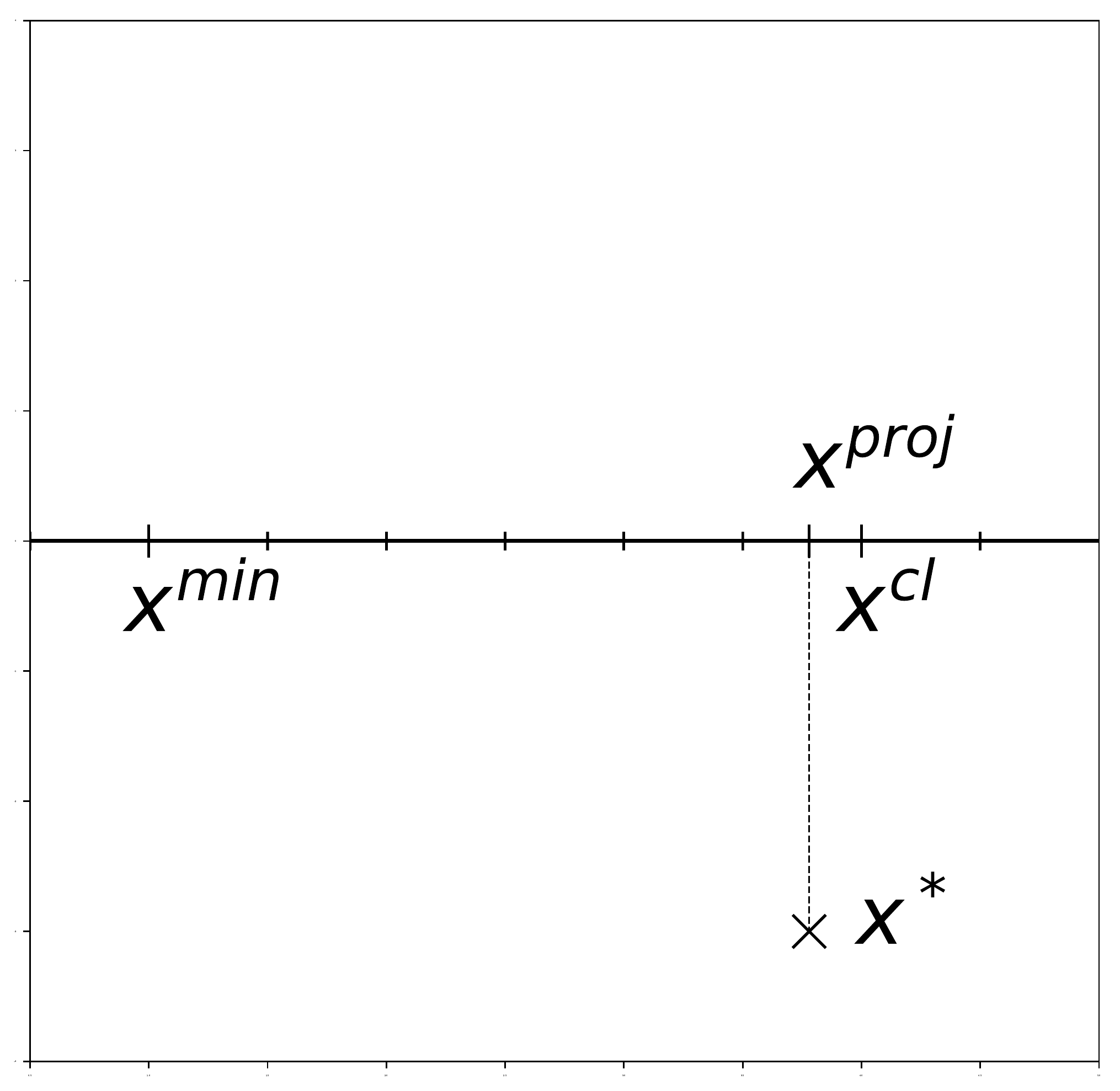}
\caption{Illustration for the proof of Theorem 2.}
\label{fig:6}
\end{wrapfigure}
\textit{Proof.} Consider $i$th iteration of inner loop (with for). Let $x^{min} = ( m_1,\ldots, m_{i-1}, b_i + j^* \cdot \nicefrac{(B_i-b_i)}{n},$
$m_{i+1}, \ldots m_d)$ , $x^{*}$ -- the real minimum, $x^{proj} = \left( m_1,\ldots, m_{i-1}, x^*_i, m_{i+1}, \ldots m_d\right)$, $x^{cl}$ -- the closest point to  $x^{proj}$, where we calculate the value of the function (see Figure \ref{fig:6}). Then the following inequality holds:
\begin{equation*} \label{noname}
    f(x^{min}) \leq f(x^{cl}).
\end{equation*}
Using \eqref{good_good}, we can note that
\begin{eqnarray*}
    \left(\frac{M}{2} + \delta(x^{min})\right)\| x^{min} - x^*\|^2_2  \quad \quad \quad \quad \quad \quad\\
    \leq  \left(\frac{M}{2} + \delta(x^{cl})\right)\| x^{cl} - x^*\|^2_2.
\end{eqnarray*}
By a definition of Euclidean norm, 
\begin{eqnarray*}
    \left(\frac{M}{2} + \delta(x^{min}) \right)\left((x^{min}_i - x^{proj}_i)^2 + \|x^{proj} - x^*\|^2\right) \quad \quad \quad \quad \quad \quad \quad\\
    \leq \left(\frac{M}{2} + \delta(x^{cl}) \right)\left((x^{cl}_i - x^{proj}_i)^2 + \|x^{proj} - x^*\|^2 \right).
\end{eqnarray*}
As we know, $|\delta(x)|$ is not greater than $\Delta$. Therefore,     \begin{eqnarray}
\label{one}
    \left(\frac{M}{2} - \Delta \right)\left((x^{min}_i - x^{proj}_i)^2 + \|x^{proj} - x^*\|^2\right) \nonumber \quad \quad \quad \quad \quad \quad \quad\\
    \leq \left(\frac{M}{2} + \Delta \right)\left((x^{cl}_i - x^{proj}_i)^2 + \|x^{proj} - x^*\|^2 \right).
\end{eqnarray}
Since $x^{cl}_i$ is the closest to $x^{proj}_i$, then
\begin{equation}
    \label{1}
    (x^{cl}_i - x^{proj}_i)^2 \leq \left( \frac{B_i - r_i}{2n}\right)^2 \leq \left( \frac{R}{2n}\right)^2.
\end{equation}
In the last inequality we use the definition of $R$ from Algorithm \ref{alg3}. 
By \eqref{1} and some simple transformation of \eqref{one}, we get
\begin{eqnarray}
\label{2}
    \left(\frac{M}{2} - \Delta \right)(x^{min}_i - x^{proj}_i)^2 
    \leq \left(\frac{M}{2} + \Delta \right)\left( \frac{R}{2n}\right)^2 + 2 \Delta \|x^{proj} - x^*\|^2 .
\end{eqnarray}
Also it is quite obvious that 
\begin{equation} \label{two}
    \|x^{proj} - x^*\|^2 \leq \sum\limits_{j \in \{1, \ldots d\} \backslash \{i\}} \left(\frac{B_j - b_j}{2}\right)^2 \leq \frac{(d - 1)R^2}{4}.
\end{equation}
Combining inequalities \eqref{2},\eqref{two},
\begin{eqnarray*}
    \left(\frac{M}{2} - \Delta \right)(x^{min}_i - x^{proj}_i)^2 
    \leq \left(\frac{M}{2} + \Delta \right)\left( \frac{R}{2n}\right)^2 +  \Delta \frac{(d - 1)R^2}{2} .
\end{eqnarray*}
Then we get the following expression
\begin{eqnarray*}
    (x^{min}_i - x^{proj}_i)^2 
    \leq \frac{\left(M + 2\Delta \right)}{\left(M - 2\Delta \right)} \left( \frac{R}{2n}\right)^2 +  \Delta \frac{(d - 1)R^2}{M - 2\Delta} .
\end{eqnarray*}
Substituting boundaries for $\Delta$ from \eqref{good_good} and $n = 15$, we have
\begin{eqnarray*}
    (x^{min}_i - x^{proj}_i)^2 
    \leq \frac{1 + \nicefrac{1}{8(d-1)} }{1 - \nicefrac{1}{8(d-1)}} \left( \frac{R}{30}\right)^2 +  \frac{R^2}{16 - \nicefrac{2}{d-1}} .
\end{eqnarray*}
Note that $d-1 \geq 1$, then 
\begin{eqnarray}
    \label{3}
    (x^{min}_i - x^{proj}_i)^2 
    \leq \left(\frac{1}{700} +  \frac{1}{14} \right)R^2 \leq \frac{51R^2}{700}.
\end{eqnarray}
Summing up \eqref{1} and \eqref{3}, we get
\begin{eqnarray*}
    |x^{min}_i - x^{cl}_i|
    \leq |x^{min}_i - x^{proj}_i| + |x^{cl}_i - x^{proj}_i| \leq \left(\frac{1}{30} + \sqrt{\frac{51}{700}}\right) R \leq \frac{R}{3}.
\end{eqnarray*}
It means that
\begin{eqnarray*}
    \frac{|x^{min}_i - x^{cl}_i|}{\max\limits_{j \in \{1,\ldots d \}} B_j - b_j} \leq \frac{1}{3}.
\end{eqnarray*}
This inequality ensures that the maximum edge length of the new cube is (at least) $\nicefrac{3}{2}$ times less than the maximum edge length of the old one.
\EndProof

\begin{corollary}
Algorithm \ref{alg3} requires $\mathcal{O}\left(\log_{\nicefrac{3}{2}}\left(\nicefrac{1}{\varepsilon}\right)\right)$ outer iterations to find a solution. Moreover, the oracle complexity of each iteration is $\mathcal{O}\left(d\right)$.
\end{corollary}
\begin{proof}

To prove the first statement, we write the simple chain for $B$ and $b$ after $T$ iteration:
\begin{equation*}
    \|B^T - b^T \|^2 \leq d \left(\max_{i \in \{1,\ldots d \}} B^T_i - b^T_i \right)^2 \leq d \left(\frac{ \max_{i \in \{1,\ldots d \}} [u_i - l_i]}{\nicefrac{3}{2}^T} \right)^2 \leq \varepsilon^2,
\end{equation*}
where $u$ and $l$ -- starting cube boundaries. This implies the required statement. 

The second statement follows immediately from the description of the algorithm.
\EndProof
\end{proof}

\textbf{Remark.} One can note that if the current $i$th variable has a small edge compared to the rest, then we essentially do the internal iteration over it in vain, because we shrink the current cube edge using the length of the maximum one. Therefore, {\tt Direction BBS} algorithm can be modified as follows: to remove the inner loop (with for) and to iterate the main loop (with while) in the direction that has the longest edge at the moment, this approach eliminates the case when we wastefully consider a small edge in the presence of a large one.

\subsection{Practical application}

\begin{wrapfigure}[18]{r}{0.5\linewidth} 
\vspace{-3ex}
\includegraphics[width =  1\linewidth]{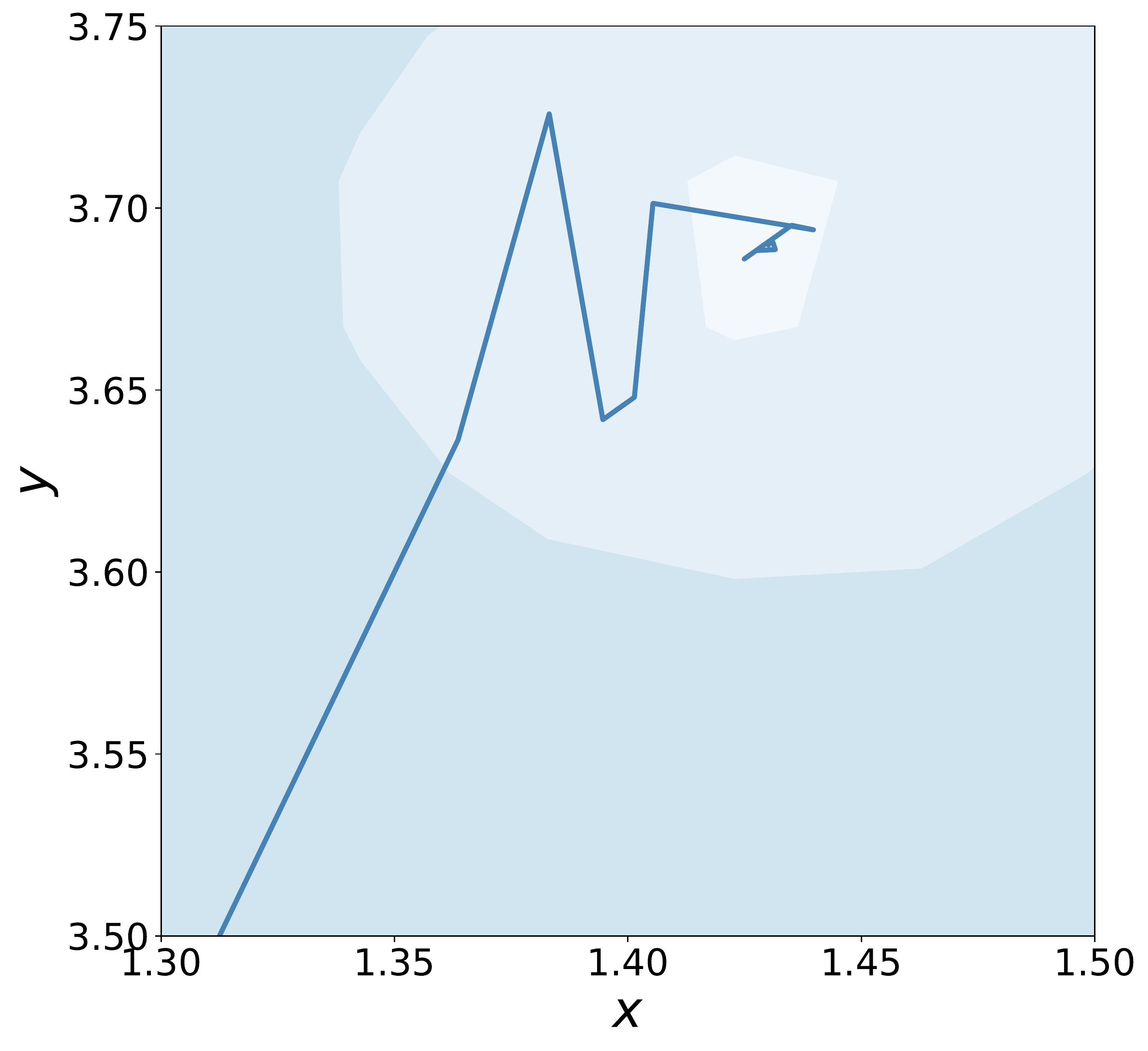}
\caption{Convergence trajectory of {\tt Direction BBS} for $d=2$.}
\label{fig:4}
\end{wrapfigure}

The function $f(x)$ from \eqref{good_good} is created in the following way: we select some point $x^*$, and also generate the values of the function $\delta (x)$ uniformly on the segment $[-\Delta ,\Delta]$ and 
independent of values in other points, then it is easy to construct $f(x)$ with $M = 20$. In the experiment, we do not restore the $\delta (x)$ function completely, we generate values only at the required points.

We start with a problem of dimension 2. The optimal point $x^* = (1.43, 3.69)$, bounds $l_i = -10, u_i = 10$ (for $i=1,2$). The convergence of the method is shown in Figure \ref{fig:4}.

Next, we used our algorithm for 10- and 100-dimensional problems with $x^* = (1, \ldots, 1)$, $l_i = -10, u_i = 10$ (for all $i$). The result of the convergence of the display in Figure \ref{fig:5}.

The results confirm the linear convergence stated in the theory.

\begin{figure}[h]
\centering
\includegraphics[width =  1\linewidth]{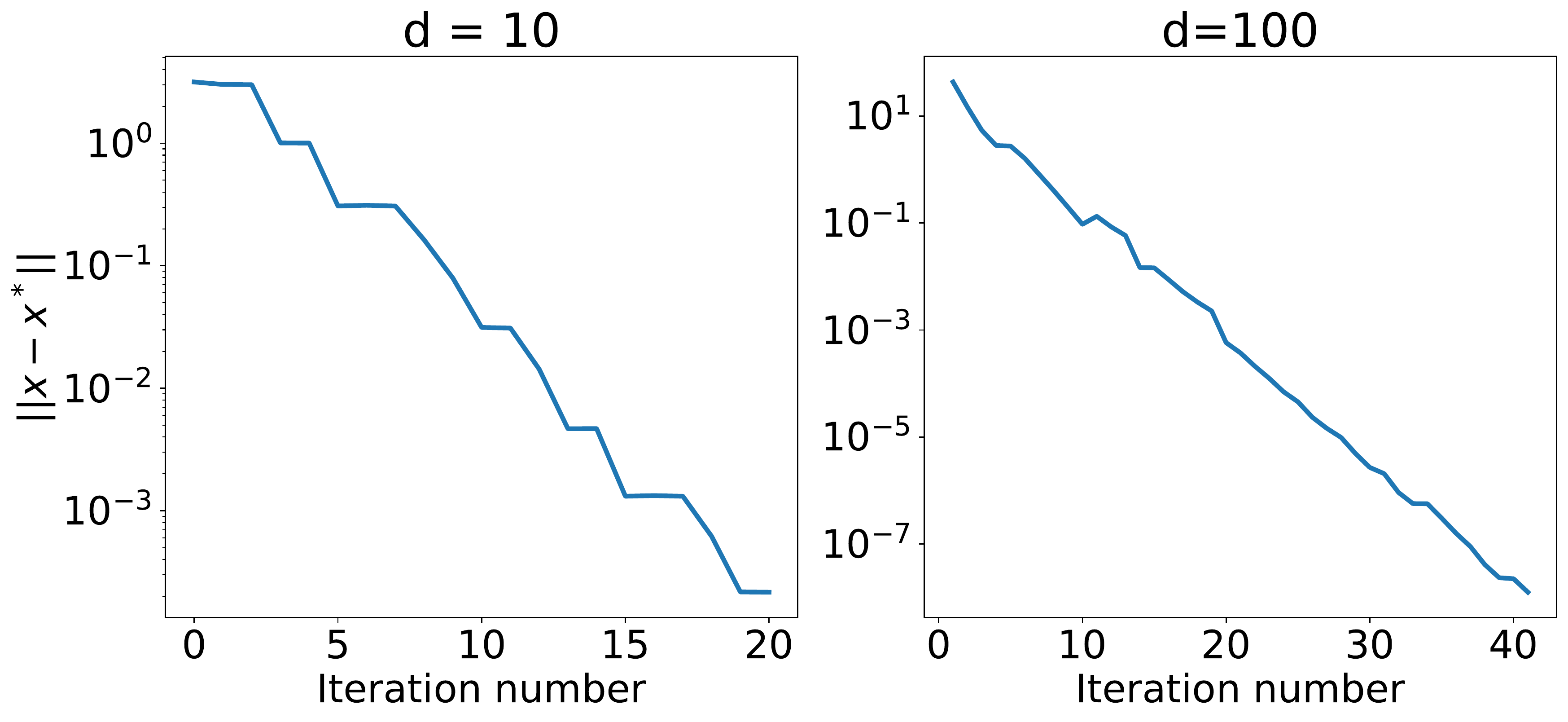}
\begin{minipage}{0.49\textwidth}
\begin{center}
~~~~~~~~~~~~~~~~~~(a)
\end{center}
\end{minipage}
\begin{minipage}{0.49\textwidth}
\begin{center}
~~~~~~(b)
\end{center}
\end{minipage}
\caption{Convergence of the {\tt Direction BBS} method in the case when the dimension of the problem is: (a) 10, (b) 100}
\label{fig:5}
\end{figure}

\vspace{\baselineskip}

Now let's move on to the second part of the paper.

\section{Another view of "good" functions}

Recall that in this section we minimize the function $\frac{1}{2}(x-x^*)^T A (x - x^*)$ with positive definite matrix $A$. One can note that our problem is $\mu$ -strongly-convex and have $L$ Lipshitz gradient. We have access to noise oracle:
\begin{equation*}
    f(x, \xi) = \frac{1}{2}(x-x^*)^TA (x-x^*) + (\xi + \delta(x)) \|x-x^* \|,
\end{equation*}
where random variable $\xi$ does not depend on the point $x$ and is generated randomly so that $\EE \xi = 0$ and $\EE \xi^2 \leq \sigma^2$, $|\delta(x)| \leq \Delta$ for all $x$. 
Then, with the help of such an oracle, one can restore the gradient using \eqref{grad_fd} in the following form:
\begin{eqnarray}
\label{grad_q}
g(x,\xi^{\pm},\tau,e) &=& n \langle A (x-x^*) , e\rangle e +
\frac{n}{2\tau} \Big((\xi^+ + \delta(x+\tau e)) \|x+\tau e-x^* \| \nonumber\\
&&- (\xi^- + \delta(x-\tau e)) \|x-\tau e-x^* \|\Big)e, 
\end{eqnarray}
where $e$ -- random vector uniformly distributed on the Euclidean sphere.
Then with this oracle we can run the classic Gradient Descent:  
    
\begin{algorithm}[H]
\caption{{\tt zoGD}}
	\label{alg4}
\begin{algorithmic}
\State 
\noindent {\bf Input:} Number of iterations $K$, parameters $\gamma_k$, $\tau_k$.
\For{$k =1,\ldots,K$}
    \State Generate independently $\xi^{\pm}_k, e_k$,
    \State $x_{k+1} = x_k - \gamma_k g(x_k,\xi^{\pm}_k,\tau_k,e_k)$.
\EndFor
\State 
\noindent {\bf Output:} $x_{K+1}$.
\end{algorithmic}
\end{algorithm}

\subsection{Theoretical analysis}

Here we prove the convergence theorem and corollaries.

\begin{theorem} The following estimate for the iteration of Algorithm \ref{alg4} is valid
\begin{eqnarray*}
\EE\left[\|x_{k+1} - x^* \|^2\right] 
&\leq& \left(1 - \gamma_k\mu + \frac{5d^2\gamma_k^2(\Delta^2+\sigma^2)}{\tau_k^2}\right)\EE\left[ \|x_{k} - x^* \|^2\right] \nonumber\\
&&+ \frac{2d\gamma_k\Delta}{\tau_k} \EE\left[\|x_k-x^* \|\right] + 2d\gamma_k\Delta + 5d^2\gamma_k^2(\Delta^2+\sigma^2).
\end{eqnarray*}
\end{theorem}

\begin{proof}
See proof in Appendix. 
\end{proof}

\begin{corollary} If $\Delta = 0$, $\gamma_k = \gamma \leq \frac{1}{5dL}$ and $\tau_k \geq \sqrt{\frac{2d \sigma^2 }{\mu L}}$ then
\begin{eqnarray*}
\EE\left[\|x_{K+1} - x^* \|^2\right] 
&\leq& \left(1 - \frac{\gamma\mu}{2}\right)^K\|x_{0} - x^* \|^2 +\frac{10d^2\gamma\sigma^2}{\mu}.
\end{eqnarray*}
Additionally, if $\gamma = \min\left\{\frac{1}{5dL}; \frac{2\ln\left(\max\left\{2; \frac{\mu^2\|z^0 - z^*\|^2 K}{20d^2\sigma^2} \right\}\right)}{\mu K}\right\}$, then
\end{corollary}
\begin{eqnarray*}
\EE\left[\|x_{K+1} - x^* \|^2\right] 
&\leq& \mathcal{\tilde O}\left(\exp\left(- \frac{\mu K}{20 d L}\right)\|x_{0} - x^* \|^2 +\frac{20d^2\sigma^2}{\mu^2 K}\right).
\end{eqnarray*}
It gives the results similar to SGD convergence. 

\subsection{Practical application}

We consider oracle \eqref{good_bad} with $n =50$, $L=100$, $\mu = 1$ and $\sigma = 1, 2, 5, 10, 20, 100$, $\Delta = 0$. In Algorithm \ref{alg4} we use constant $\gamma = \frac{1}{dL}$ and $\tau = \sqrt{\frac{2d \sigma^2 }{\mu L}}$. See convergence of Algorithm \ref{alg4} with different $\sigma$ on Figure \ref{fig:8}.

\begin{figure}[h!]
\centering
\begin{minipage}{0.98\textwidth}
\includegraphics[width =  \textwidth]{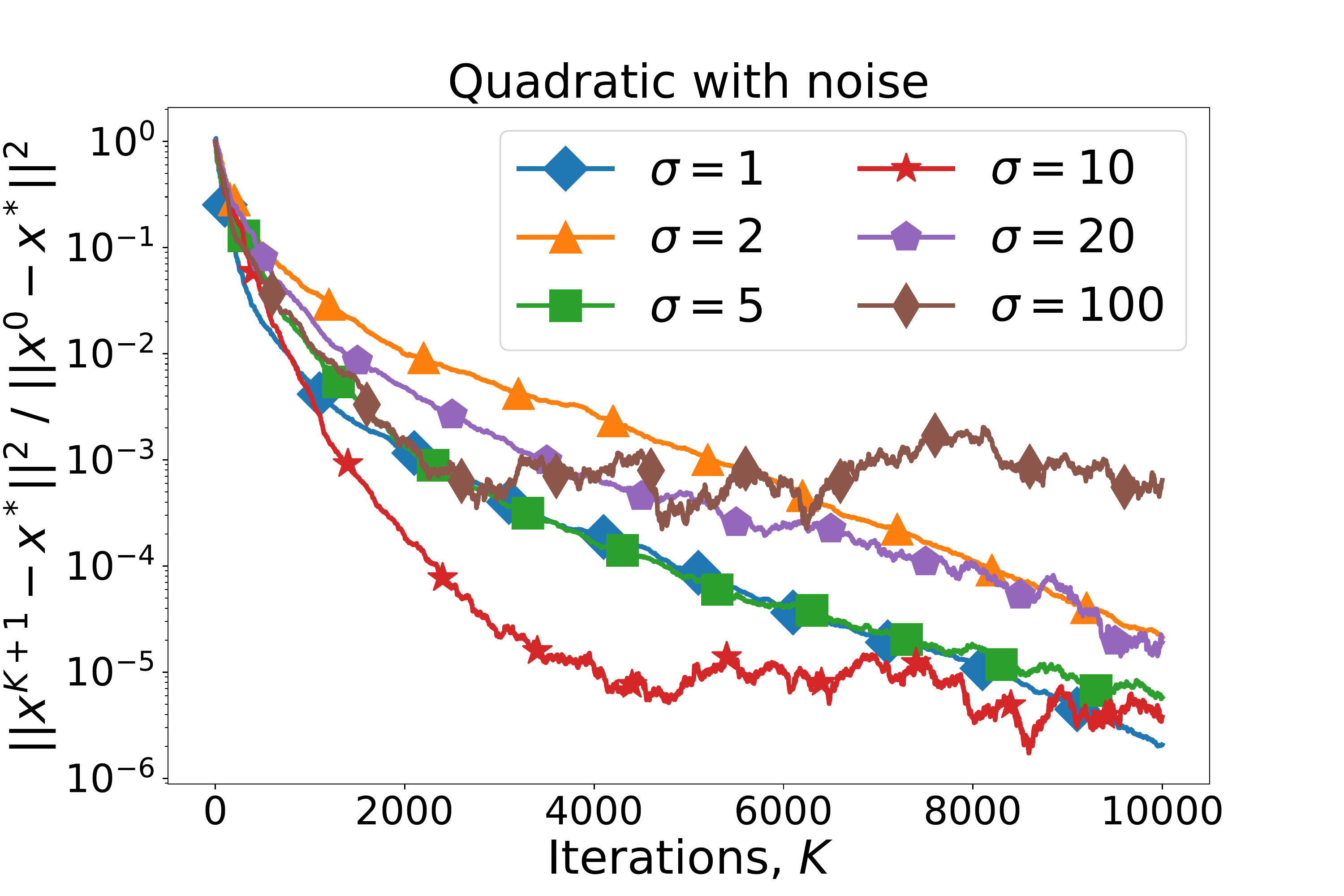}
\end{minipage}%
\caption{Convergence of Algorithm \ref{alg4} with oracle \eqref{good_bad} at different noise levels.}
\label{fig:8}
\end{figure}

\section{Conclusion} \label{concl}

In this paper, we presented methods for a non-convex problem that can be approximated by a symmetric parabolic function. Our first method is suitable for a larger class of problems, but the complexity of its iterations grows exponentially depending on the dimension of the problem. The second method is intended for a restricted set of functions, but it does not so dramatically depend on the dimension.

For future work, we highlight the following areas. It is important to understand whether the results obtained are valid for the problem approximated by arbitrary quadratic functions. It is also interesting to consider other types of functions (not only parabolic and quadratic), as well as their various combinations.

In this setting, the noise level depends on the distance to the solution. This concept of noise is new, it is usually assumed that the noise (or its second moment) is uniformly bounded \cite{akhavan2020exploiting,NIPS2016_186fb23a}. 

We also presented the minimization of a quadratic function using Gradient Descent with a zero-order inexact oracle. Moreover, the noise (inexactness) in this oracle is proportional to the distance to the solution.

In further work, we would like to shift this analysis to arbitrary functions (not necessarily quadratic), and also consider other concepts of noise depending on the value of the function/distance to the solution.

\bibliographystyle{splncs04}
\bibliography{literature}

\appendix



\section{Proof for Section 4}

\begin{proof} We start with the step of our method:
\begin{eqnarray*}
\|x_{k+1} - x^* \|^2 &=& \|x_{k} - \gamma_k g(x_k,\xi^{\pm}_k,\tau_k,e_k) - x^* \|^2 \nonumber\\
&=& \|x_{k} - x^* \|^2 - 2\gamma_k \langle g(x_k,\xi^{\pm}_k,\tau_k,e_k), x_k  - x^*\rangle  \nonumber\\
&&+ \gamma_k^2 \|g(x_k,\xi^{\pm}_k,\tau_k,e_k)\|^2.
\end{eqnarray*}
Taking the full expectation and taking into account that $z_{k} - z^*$ does not depend on $e_k, \xi_k$, we get
\begin{eqnarray}
\label{temp1}
\EE\left[\|x_{k+1} - x^* \|^2\right] &=& \EE\left[ \|x_{k} - x^* \|^2\right] - 2\gamma_k \EE\left[\langle g(x_k,\xi^{\pm}_k,\tau_k,e_k), x_k  - x^*\rangle \right]  \nonumber\\
&&+ \gamma_k^2 \EE\left[\|g(x_k,\xi^{\pm}_k,\tau_k,e_k)\|^2\right] \nonumber\\
&=& \EE\left[ \|x_{k} - x^* \|^2\right] - 2\gamma_k \EE\left[\langle \EE_{\xi_k,e_k}\left[g(x_k,\xi^{\pm}_k,\tau_k,e_k)\right], x_k  - x^*\rangle \right]  \nonumber\\
&&+ \gamma_k^2 \EE\left[\|g(x_k,\xi^{\pm}_k,\tau_k,e_k)\|^2\right].
\end{eqnarray}
We use that $\EE[d\langle s, e\rangle e] = s$. Next we need to estimate $\EE\left[\langle g(x_k,\xi^{\pm}_k,\tau_k,e_k), x_k  - x^*\rangle \right]$ and $\EE\left[\|g(x_k,\xi^{\pm}_k,\tau_k,e_k)\|^2\right]$. We start with $\EE\left[\langle g(x_k,\xi^{\pm}_k,\tau_k,e_k), x_k  - x^*\rangle \right]$:
\begin{eqnarray}
\label{temp2}
\EE_{\xi_k,e_k}\left[g(x_k,\xi^{\pm}_k,\tau_k,e_k)\right] &=& \EE_{e_k}\left[d \langle A (x_k-x^*) , e_k\rangle e_k\right]  \nonumber\\
&&+
\frac{d}{2\tau_k} \EE_{\xi_k,e_k}\Big[\Big((\xi^+_k + \delta(x_k+\tau_k e_k)) \|x_k+\tau_k e_k-x^* \| \nonumber\\
&&- (\xi_k^- + \delta(x_k-\tau_k e_k)) \|x_k-\tau_k e_k-x^* \|\Big)e_k\Big] \nonumber\\
&=& A (x_k-x^*) +
\frac{d}{2\tau_k} \EE_{e_k}\Big[\Big(\delta(x_k+\tau_k e_k)\|x_k+\tau_k e_k-x^* \| \nonumber\\
&&- \delta(x_k-\tau_k e_k)\|x_k-\tau_k e_k-x^* \|\Big)e_k\Big].
\end{eqnarray}
 Let's  work with $\EE\left[\|g(x_k,\xi^{\pm}_k,\tau_k,e_k)\|^2\right]$:
\begin{eqnarray*}
\EE\left[\|g(x_k,\xi^{\pm}_k,\tau_k,e_k)\|^2\right] &=& \EE\Big[\EE_{e_k}\big[\|d \langle A (x_k-x^*) , e_k\rangle e_k \|^2\big] \nonumber\\
&&+ \frac{d}{2\tau_k} \Big((\xi_k^+ + \delta(x_k+\tau_k e_k)) \|x_k+\tau_k e_k-x^* \| \nonumber\\
&&- (\xi_k^- + \delta(x_k-\tau_k e_k)) \|x_k-\tau_k e_k-x^* \|\Big)e_k\|^2\Big] \nonumber\\
&\leq& 5n^2\EE\left[\EE_{e_k}\big[\|d \langle A (x_k-x^*) , e_k\rangle e_k \|^2\big]\right] \nonumber\\
&&+ \frac{5d^2}{4\tau_k^2} \EE\left[(\xi_k^+)^2\|x_k+\tau e_k-x^* \|^2\|e_k\|^2\right]
\nonumber\\
&&+ \frac{5d^2}{4\tau_k^2} \EE\left[\delta_k^2(x_k+\tau_k e_k)\|x_k+\tau e_k-x^* \|^2\|e_k\|^2\right]
\nonumber\\
&&+ \frac{5d^2}{4\tau_k^2} \EE\left[(\xi_k^-)^2\|x-\tau_k e_k-x^* \|^2\|e_k\|^2\right]
\nonumber\\
&&+ \frac{5d^2}{4\tau_k^2} \EE\left[\delta^2(x_k-\tau_k e_k)\|x_k-\tau_k e_k-x^* \|^2\|e_k\|^2\right].
\end{eqnarray*}
Next we use Lemma B.10 from \cite{bogolubsky2016learning}: $\EE[|\langle s, e\rangle|^2] = \frac{1}{d}\|s\|^2$
for some vector $s$ and independent $e$ -- random vector uniformly distributed on the Euclidean sphere.
\begin{eqnarray}
\label{temp3}
\EE\left[\|g(x_k,\xi^{\pm}_k,\tau_k,e_k)\|^2\right] &\leq& 5d\EE\left[\|A (x_k-x^*) \|^2\right] \nonumber\\
&&+ \frac{5d^2}{4\tau_k^2} \EE\left[\EE_{\xi_k}\left[(\xi_k^+)^2\right]\|x_k+\tau_k e_k-x^* \|^2\right]
\nonumber\\
&&+ \frac{5d^2\Delta^2}{4\tau^2} \EE\left[\|x_k+\tau_k e_k-x^* \|^2\right]
\nonumber\\
&&+ \frac{5d^2}{4\tau_k^2} \EE\left[\EE_{\xi_k}\left[(\xi_k^-)^2\right]\|x_k-\tau_k e_k-x^* \|^2\right]
\nonumber\\
&&+ \frac{5d^2\Delta^2}{4\tau_k^2} \EE\left[\|x_k-\tau_k e_k-x^* \|^2\right] \nonumber\\
&\leq& 5n\EE\left[\|A (x_k-x^*)\|^2\right] \nonumber\\
&&+ \frac{5d^2 \sigma^2}{2\tau_k^2} \EE\left[\|x_k-x^* \|^2 + \|\tau_k e_k\|^2\right]
\nonumber\\
&&+ \frac{5d^2\Delta^2}{2\tau_k^2} \EE\left[\|x_k-x^* \|^2 + \|\tau_k e_k\|^2\right]
\nonumber\\
&&+ \frac{5d^2 \sigma^2}{2\tau_k^2} \EE\left[\|x_k-x^* \|^2 + \|\tau_k e_k\|^2\right]
\nonumber\\
&&+ \frac{5d^2\Delta^2}{2\tau_k^2} \EE\left[\|x_k-x^* \|^2 + \|\tau_k e_k\|^2\right]\nonumber\\
&\leq& 5d\EE\left[\|A (x_k-x^*)\|^2\right] + \frac{5d^2(\Delta^2+\sigma^2)}{\tau_k^2}\EE\left[\|x_k-x^* \|^2 \right] \nonumber\\
&& + 5d^2(\Delta^2+\sigma^2).
\end{eqnarray}
Then we combine \eqref{temp1}, \eqref{temp2} and \eqref{temp3}:
\begin{eqnarray*}
\EE\left[\|x_{k+1} - x^* \|^2\right] 
&\leq& \EE\left[ \|x_{k} - x^* \|^2\right] - 2\gamma_k \EE\left[\langle A (x_k-x^*), x_k  - x^*\rangle \right] \nonumber\\
&&- \frac{d\gamma_k}{\tau_k} \EE\Big[ \langle\EE_{e_k}\Big[\Big(\delta(x_k+\tau_k e_k)\|x_k+\tau_k e_k-x^* \| \nonumber\\
&&- \delta(x_k-\tau_k e_k)\|x_k-\tau_k e_k-x^* \|\Big)e_k\Big], x_k  - x^*\rangle \Big]  \nonumber\\
&&+ 5d\gamma_k^2 \EE\left[\|A (x_k-x^*)\|^2\right] + \frac{5d^2\gamma_k^2(\Delta^2+\sigma^2)}{\tau_k^2}\EE\left[\|x_k-x^* \|^2 \right] \nonumber\\
&& + 5d^2\gamma_k^2(\Delta^2+\sigma^2).
\end{eqnarray*}
With $\nabla f(x_k) = A(x_k - x^*)$, we have
\begin{eqnarray}
\label{temp4}
\EE\left[\|x_{k+1} - x^* \|^2\right] 
&\leq& \left(1 + \frac{5d^2\gamma_k^2(\Delta^2+\sigma^2)}{\tau_k^2}\right)\EE\left[ \|x_{k} - x^* \|^2\right] \nonumber\\
&&- 2\gamma_k \EE\left[\langle \nabla f(x_k), x_k  - x^*\rangle \right] + 5d\gamma_k^2 \EE\left[\|\nabla f(x_k)\|^2\right] \nonumber\\
&&+ \frac{n\gamma_k}{\tau_k} \EE\Big[ \Big\|\EE_{e_k}\Big[\Big(\delta(x_k+\tau_k e_k)\|x_k+\tau_k e_k-x^* \| \nonumber\\
&&- \delta(x_k-\tau_k e_k)\|x_k-\tau_k e_k-x^* \|\Big)e_k\Big] \Big\| \cdot \| x_k  - x^*\| \Big]  \nonumber\\
&& + 5d^2\gamma_k^2(\Delta^2+\sigma^2).
\end{eqnarray}
We work with $ \Upsilon_k = \Big\|\EE_{e_k}\Big[\Big(\delta(x_k+\tau_k e_k)\|x_k+\tau_k e_k-x^* \|- \delta(x_k-\tau_k e_k)\|x_k-\tau_k e_k-x^* \|\Big)e_k\Big] \Big\|$:
\begin{eqnarray}
\Upsilon_k 
&\leq& \EE_{e_k}\Big[|\delta(x_k+\tau_k e_k)|\cdot\|x_k+\tau_k e_k-x^* \| + |\delta(x_k-\tau_k e_k)| \cdot\|x_k-\tau_k e_k-x^* \|\Big] \nonumber\\
&\leq& \Delta\EE_{e_k}\Big[\|x_k+\tau_k e_k-x^* \| + \|x_k-\tau_k e_k-x^* \|\Big]\nonumber\\
&\leq& 2\Delta\EE_{e_k}\Big[\|x_k-x^* \| + \tau_k\Big].
\end{eqnarray}
Connecting with \eqref{temp4}, we get
\begin{eqnarray*}
\EE\left[\|x_{k+1} - x^* \|^2\right] 
&\leq& \left(1 + \frac{5d^2\gamma_k^2(\Delta^2+\sigma^2)}{\tau_k^2}\right)\EE\left[ \|x_{k} - x^* \|^2\right] \nonumber\\
&&- 2\gamma_k \EE\left[\langle \nabla f(x_k), x_k  - x^*\rangle \right] + 5d\gamma_k^2 \EE\left[\|\nabla f(x_k)\|^2\right] \nonumber\\
&&+ \frac{2d\gamma_k\Delta}{\tau_k} \EE\left[\|x_k-x^* \|\right] + 2d\gamma_k\Delta + 5d^2\gamma_k^2(\Delta^2+\sigma^2).
\end{eqnarray*}
Using $\mu$-strong convexity and $L$-smoothness of $f$, we have
\begin{eqnarray*}
\EE\left[\|x_{k+1} - x^* \|^2\right] 
&\leq& \left(1 - \gamma_k\mu + \frac{5d^2\gamma_k^2(\Delta^2+\sigma^2)}{\tau_k^2}\right)\EE\left[ \|x_{k} - x^* \|^2\right] \nonumber\\
&& - 2\gamma_k (1 - 5d\gamma_k L) \EE\left[f(x_k)  - f(x^*) \right] \nonumber\\
&&+ \frac{2d\gamma_k\Delta}{\tau_k} \EE\left[\|x_k-x^* \|\right] + 2d\gamma_k\Delta + 5d^2\gamma_k^2(\Delta^2+\sigma^2).
\end{eqnarray*}
With $\gamma_k \leq \frac{1}{5dL}$
\begin{eqnarray*}
\EE\left[\|x_{k+1} - x^* \|^2\right] 
&\leq& \left(1 - \gamma_k\mu + \frac{5d^2\gamma_k^2(\Delta^2+\sigma^2)}{\tau_k^2}\right)\EE\left[ \|x_{k} - x^* \|^2\right] \nonumber\\
&&+ \frac{2d\gamma_k\Delta}{\tau_k} \EE\left[\|x_k-x^* \|\right] + 2d\gamma_k\Delta + 5d^2\gamma_k^2(\Delta^2+\sigma^2).
\end{eqnarray*}
\EndProof
\end{proof}


\end{document}